\documentclass[11pt,reqno]{amsart}

\usepackage{amsfonts, amssymb, amsmath}
\usepackage{amsthm}                               
\usepackage{setspace}                                                         
\usepackage{geometry}                                                         
\usepackage{color}                                                            
\usepackage{graphicx}                                                         
\usepackage{slashed}                                                          
\usepackage{hyperref}
\usepackage{float}
\usepackage{verbatim}
\usepackage{upgreek}
\usepackage{mathtools}
\usepackage[capitalise]{cleveref}

\newtheorem{theorem}{Theorem} 	      	      	                              
\newtheorem{definition}[theorem]{Definition} 	      	      	                
\newtheorem{remark}[theorem]{Remark}                                          
\newcommand{\thistheoremname}{}
\newtheorem*{theorem*}{Theorem}
\newtheorem*{genericthm*}{\thistheoremname}
\newenvironment{namedthm*}[1]
  {\renewcommand{\thistheoremname}{#1}%
   \begin{genericthm*}}
  {\end{genericthm*}}

\numberwithin{equation}{section}                                              
\numberwithin{theorem}{section}                                               
\numberwithin{figure}{section}                                                

\newcommand{\mf}[1]{\mathfrak{#1}}                                            
\newcommand{\mc}[1]{\mathcal{#1}}                                             
\newcommand\numberthis{\addtocounter{equation}{1}\tag{\theequation}}
\newcommand{\mb}[1]{\mathbb{#1}}

\newcommand{\N}{\mathbb{N}}                                                   
\newcommand{\R}{\mathbb{R}}                                                   
\newcommand{\rd}{{\rm d}}
\newcommand{\U}{\mathfrak{U}}
\newcommand{\D}{\mathfrak{D}}
\newcommand{\Rmn}{\mathbb{R}^{m+n}}
\newcommand{\X}{\mathcal{X}}

\newcommand{\nasla}{\slashed{\nabla}}                                         

\newcommand{\nt}{\numberthis}

\newcommand{\p}{\partial}
\newcommand{\nb}{\nabla}

\allowdisplaybreaks
\begin{document}

\title[Uniqueness for ultrahyperbolic equations]{Uniqueness of solution for ultrahyperbolic equations \\with lower order terms}

\author{Vaibhav Kumar Jena}
\address{Department of Mathematics, Indian Institute of Science,\\
Bangalore, 560012, Karnataka, India.}
\email{vaibhavkuma1@iisc.ac.in}

\begin{abstract}
In this article, we prove a variety of uniqueness results for ultrahyperbolic equations with general space and time dependent lower order terms. We address the problem of determining uniqueness of solutions from boundary data as well as when the data is prescribed on an interior subset. Furthermore, we also present the case when the domain may change with respect to any one time component and obtain analogous results. Our main tool for this purpose is Carleman estimate. We obtain different uniqueness results depending on the location of the reference point for the Carleman estimate relative to the domain.
\end{abstract}

\maketitle

\section{Introduction}

In this article, we are going to address some uniqueness results for ultrahyperbolic equations. Such an equation is posed on domains in \(\mathbb{R}^m \times \mathbb{R}^n\), for \( m,n \in \mathbb{N}\).  When we consider the case of $\R^1\times\R^n$ operators, we have the standard wave operator given by
\[ \mc{L} := - \p_t^2 + \Delta_x = -\p_t^2 + \sum_{j=1}^n \p_{x_j}^2.\]
A natural generalisation of the above expression to the case of multiple time dimensions is given by the ultrahyperbolic operator, which has the form
\[\square := -\Delta_t + \Delta_x = - \sum_{i=1}^m \p_{t_i}^2 + \sum_{j=1}^n \p_{x_j}^2 = - \p_{t_1}^2 - \p_{t_2}^2 - \cdots - \p_{t_m}^2 + \p_{x_1}^2 + \p_{x_2}^2 + \cdots + \p_{x_n}^2. \]
We are going to show suitable uniqueness results for the ultrahyperbolic equation  with lower-order terms that depend on both space and time. We address the case when data is prescribed on the boundary and also when it is given on an interior subset. Moreover, we will also consider the case when the domain is time dependent.

\subsection{Literature}

The main motivation for studying this problem was to study its controllability properties. However, the first obstacle in this direction is that such equations are usually not well-posed. Hence, as a first step, we decided to study a part of the well-posedness issue, that is, uniqueness of solution.

Ultrahyperbolic equations have generated significant interest in the Physics community. In particular, the concept of multiple temporal dimensions is important in the field of string theory. We refer to the article by \emph{Tegmark} \cite{MR1441890} for a good introduction on this topic. For further reading on applications of the ultrahyperbolic equation in Physics, see \cite{MR1878197, foster2010physicstimedimensions, MR3736101} and the references therein.

In the mathematical community there has been less development, primarily due to the issue of well-posedness.  The first work in this direction was by \emph{John} \cite{MR1546052}, who showed that the set of solutions of the ultrahyperbolic equation in $\R^2\times\R^2$, in appropriate spaces, is actually the range of an integral transform known as the X-ray transform, up to a suitable factor.
Next, \emph{Courant-Hilbert} \cite{MR140802} used the mean value theorem of \emph{\'Asgeirsson} \cite{MR1513094} to show that existence of solution for the ultrahyperbolic equation fails if the initial data is not properly posed. \emph{H\"ormander} \cite{MR1851002} later generalised the \'Asgeirsson's mean value theorem to solutions of the inhomogeneous equation.

The work by \emph{Owens} \cite{MR19825} gives sufficient conditions for uniqueness of solutions for the homogeneous equation for mixed problems having elliptic and hyperbolic nature. Also see other work by the same author for some related uniqueness problems \cite{MR49460, MR120453}. In \emph{Murray-Protter} \cite{MR348288}, the authors study the asymptotic behaviour of the ultrahyperbolic equation with lower order terms. As for the uniqueness problem, they only solve the problem with a zeroth order term, they do not consider the case when a first order term is present. We point out that, this proof involves the use of a Carleman-type estimate. In \emph{Diaz-Young} \cite{MR283394}, the authors consider an equation of the type
\[ \sum_{i=1}^m \p_{t_i}^2 y - \sum_{j,k=1}^n \p_{x_k} (a_{jk} \p_{x_j} y) + c y = 0, \]
and give a necessary and sufficient condition for uniqueness that relates the temporal dimension `$m$' with a certain eigenvalue problem posed in $G$. Also see \cite{MR364887} for a related result involving multiple singular lines in the equation. The work by \emph{Burskii-Kirichenko} \cite{MR2432862} solves a uniqueness problem for the equation, without any lower order terms, in a ball by giving a condition in terms of zeroes of Jacobi polynomials.

Also see \emph{Craig-Weinstein} \cite{MR2534852}, where the authors show that the initial value problem is well posed, subject to a non-local constraint, when the initial data is given on a co-dimension one hypersurface. They also show non-uniqueness results when initial data is given on higher co-dimension surfaces. 

Another set of references has to do with inverse problems for ultrahyperbolic equations: see \cite{MR1861908,MR3227743,MR4068237,MR4669095}. For some results related to ultrahyperbolic equations using representation theory see \cite{MR2020552,MR2998915}. For work on ultrahyperbolic Schr\"odinger equation we refer to \cite{MR4038128}.

\subsection{Main results}
To present our main result, we provide some definitions. 
\begin{definition}
Let $T>0$ and let $G$ be a hypercube of side $T$ in $\R^m$, that is, 
\[G:= \{ (t_1,\ldots,t_m) \in \R^m | -T < t_i < T, \text{ for } 1 \leqslant i \leqslant m \} \subset \R^m.\]
Let $\Omega \subset \R^n$ be a non-empty, open, and bounded subset. Then we consider a domain of the form $\U := G \times \Omega \subset \Rmn$.
\end{definition}

In this work, we are exclusively going to deal with the boundary of the spatial part of $\U$.
Hence, for convenience we will write $\p\U:= G \times \p\Omega$. We will not be dealing with the temporal boundary of $\U$ in this work. This is because our assumption on $T$ ensures that we do not encounter the temporal part of the boundary of $\U$, see \cref{rem_Tlarge}. Moreover, as we are not looking to address initial value problems this is justified. 

Now we will state a less precise version of our main result as it is more insightful. It is also easier to represent this result visually. As per the author's knowledge, the most general case that can be represented in images is the case when $t\in \R$ and $x \in \R^2$; see \cref{fig1} and \cref{fig2} for the boundary case. For higher dimensions, the figures are too complex to visualise.

\begin{definition}
Fix $ p:= (t(p),x(p)) \in \R^{m+n}$, and define the function $f_p$ and the region $\D_p$ as
\[f_p := \frac{1}{4} \left( |x-x(p)|^2 - |t-t(p)|^2 \right), \quad \mf{D}_p := \{ f_p > 0 \}.\]
Let $\mc{N}$ denote the outward pointing unit normal of $\U$. Define the region 
\[\Gamma_p := \p \U \cap \D_p \cap \{ \mc{N} f_p > 0 \}.\]
and let $\Theta_p$ be any neighbourhood of $\bar{\Gamma}_p$ in $\p\U$. 
For any $y \in \R^n$ and $\sigma > 0 $, we define the $\sigma$-neighbourhood of $y$ in $\R^n$ as follows
\[\mc{O}_\sigma(y) := \{ y_1 \in \R^n : \|y_1-y\|_{\R^n} < \sigma\} \subset \R^n.\]
Then let us define the following
\[\mc{O}_\sigma (\Gamma_p) := \bigcup_{(\tau,y) \in \Gamma_p} \left( \{ \tau \} \times \mc{O}_\sigma(y) \right) \subset \Rmn.\]
Finally, we define the interior set $W_p \subset \U$ to be any subset that satisfies $W_p \supset \overline{\mc{O}_\sigma (\Gamma_p)} \cap \U$.
\end{definition}
\begin{definition}
Let $\X \in C^\infty(\bar{\U};\Rmn) $ be such that 
$\X$ identically vanishes on $\U\cap \p\D_p$. Furthermore, let $V \in C^\infty (\bar{\U})$, $F \in C^\infty(\bar{\U})$ and $H \in C^\infty(\bar{G})$. 
We say that $z$ is a solution of the ultrahyperbolic equation if $z \in C^2(\U)$ and
\begin{equation}\label{eq_main_z_A}
\begin{rcases}
\square z + \nabla_\mc{X} z + V(t,x) z = F(t,x), \qquad & \text{in } \U,\\
z(t,0) = H(t) , \qquad & \text{on } \p\U.
\end{rcases}
\end{equation}
\end{definition}

\begin{theorem} \label{thm_prelim}
Let $T>0$ be sufficiently large. Let $z_1,z_2 \in C^2(\U)$ be two solutions of system \eqref{eq_main_z_A}. Then we have the following three cases for uniqueness
\begin{enumerate}
\item If $\mc{N} z_1 = \mc{N} z_2$ on $\Theta_p$, then $z_1=z_2$ on $\U \cap \D_p$.

\item If $(z_1,\nb_t z_1 ) = (z_2,\nb_t z_2 )$ on $W_p$, then $z_1=z_2$ on $\U \cap \D_p$.

\item If $(z_1,\nb_x z_1 ) = (z_2,\nb_x z_2 )$ on $W_p$, then $z_1=z_2$ on $\U \cap \D_p$.

\end{enumerate}
\end{theorem}
The precise version of the above result is given in \cref{thm_uc_bdry_Dp} for the boundary case and \cref{thm_uc_int} for the interior cases. We now discuss a bit about the assumptions of the above theorem and the proof techniques.

\begin{figure}
    \centering
\includegraphics[height=5cm]{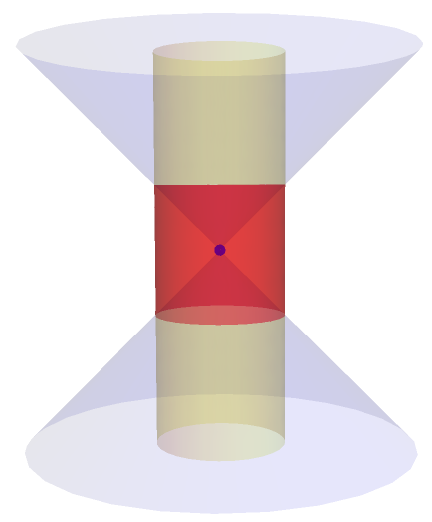} \hspace{5mm}
\includegraphics[height=5cm]{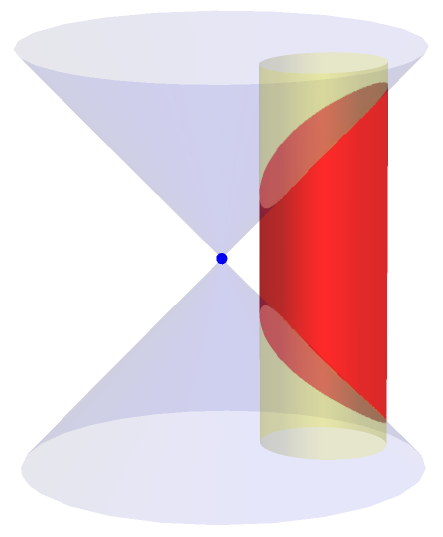} \hspace{5mm}
\includegraphics[height=5cm]{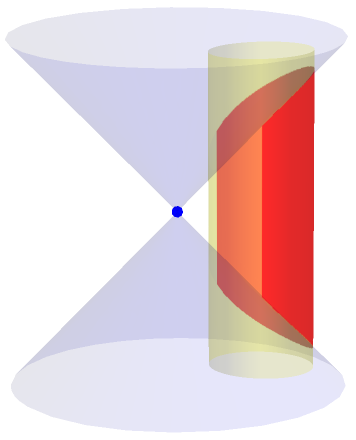}

\caption{This set of figures is for the time static domain and the case $(t,x) \in \R \times \R^2$. In all three cases, the blue vertex denotes the point $p:=(t(p),x(p)$ and the boundary of the domain $\U$, given by $\p\U$, is represented by combining the yellow and the red shaded regions. In the first image, we have $p \in \U$ and the red region denotes both $\p\U\cap\D_p$ and $\Gamma_p$, as they coincide. In the second image, the point $p \notin \bar{\U}$ and the red region denotes $\p\U \cap \D_p$. Finally, in the third image the point $p \notin \bar{\U}$ and the red region now denotes $\Gamma_p$, which is a proper subset of $\p\U\cap \D_p$.}
    \label{fig1}
\end{figure}

\begin{remark} \label{rem_Tlarge}
Let us explain the reasoning behind the assumption on $T$ being sufficiently large. In our work the key region of interest is $\U \cap \D_p$. Since we are not given any information about the ``initial data" for $z$ we want to avoid any contributions coming from the temporal part of the boundary of the domain. That is, we want the boundary of $\U \cap \D_p$ to be made of $\U \cap \p\D_p$ and $\p\U \cap \D_p$. This is ensured by assuming $T$ to be sufficiently large. In particular, if we let $R:= \sup_{\U\cap\D_p} |x-x(p)|$ and take $T>R$, our result holds.
\end{remark}

The main tool to prove our uniqueness result is Carleman estimate \cite{MR334}, which are integral inequalities used to show suitable unique continuation properties for PDEs. These estimates have applications in many areas--namely in control theory \cite{MR1809954, MR1406566, MR1290492, MR2361985, MR3964826, MR4450884, MR4797140} and inverse problems \cite{MR3460049, MR1804796}. We are going to use the Carleman estimate from \cite{MR4314050}, which proves a boundary Carleman estimate for ultrahyperbolic operators, given in \cref{thm.carl_bdry}. We will use this estimate to prove the uniqueness from the boundary data. Moreover, using \cref{thm.carl_bdry}, we derive an interior Carleman estimate \cref{thm_carl_int}, that is subsequently used to prove uniqueness of solution from interior data.

\begin{remark}
We need $\X$ to vanish on $\U\cap\p\D_p$, that is,
\begin{equation} \label{assump_X}
\X (t,x) = 0, \quad \text{for all } (t,x) \in \U \cap \{|t-t(p)|^2=|x-x(p)|^2\}.
\end{equation}
This is due to a technical issue while proving uniqueness of solution via the Carleman estimate. In particular, on the boundary of the cone given by $\p\D_p$, $f_p$ identically vanishes, which implies that our Carleman weight itself vanishes on $\p\D_p$. Hence, close to $\p\D_p$, we cannot appropriately absorb the contribution of the first order terms coming from $\nb_\X z$. However, under assumption \eqref{assump_X}, we do not get any such bad contribution due to $\X$ near $\p\D_p$. This point becomes clear when we look at the analysis after \cref{eq_uni_pf_3}.
\end{remark}

The key inspiration for the region $\Gamma_p$ is an analogous condition present in the works of \emph{Ho} \cite{MR838598} and \emph{Lions} \cite{MR931277}, which address boundary controllability of (free) wave equations. This region was generalised to wave equations with space-time dependent lower order terms, on general time dependent domains, in the work of \cite{MR3964826}. 

The region $\Gamma_p$ given above is different from the region $\Gamma_p^\varepsilon$ (depending on a parameter $\varepsilon>0$) we have in the precise results \cref{thm_uc_bdry_Dp} and its precursor \cref{thm.carl_bdry}. This is because the level sets of the function $f_p$ barely fail to be pseudoconvex with respect to the ultrahyperbolic operator; pseudoconvexity is a necessary criterion for proving a Carleman estimate. Thus, one needs to work with a perturbation of $f_p$, which leads to the different functional used to define $\Gamma_p^\varepsilon$. This perturbation of $f_p$ is used to define the Carleman weight in \eqref{eq.carleman_weight}; see \cite{MR4314050} for a detailed discussion. Our precise result is better than \cref{thm_prelim} because we only need to know the data on $\Gamma_p^\varepsilon$ to get uniqueness, rather than $\Theta_p$, which is a proper superset of $\Gamma_p^\varepsilon$. The same correspondence holds for the interior case, $W_p$ in \cref{thm_prelim} relates to $W_p^\varepsilon$ in \cref{thm_carl_int} and \cref{thm_uc_int}. The crucial point is that $\Gamma_p$ and $\Gamma_p^\varepsilon$ can be made arbitrarily close to each other by choosing suitable $\varepsilon$; see \cref{rem_compare}. We avoid presenting the result with the optimal regions $\Gamma_p^\varepsilon$ and $W_p^\varepsilon$ as it would require providing technical definitions and it is also difficult to represent this region visually.

The uniqueness we obtain is restricted to $\U \cap \D_p$, which is the part of the domain that lies in the exterior of the null cone centred at $(t(p),x(p))$. When we compare this with wave equations, such a uniqueness result would mean that we have uniqueness for waves in the whole domain, because waves are already known to be uniquely determined in the region inside the cone \cite{MR1625845}. However, such a result does not hold for the ultrahyperbolic equation due to the absence of a well posedness theory. In particular, we are lacking an analogue of energy estimate for \eqref{eq_main_z_A}. In the presence of a suitable energy estimate one can use similar ideas as in \cite{MR4450884, MR4797140} to get uniqueness in whole of $\U$. However, the perturbation of $f_p$ that we use, is pseudoconvex even inside the null cone.

For the reader's convenience, we initially present the case of the time static domain $\U=G\times \Omega$. In \cref{sec_timedep} we will generalise the domain $\U$ to time dependent domains. In particular, we will consider the case when the domain may change with respect to any one time component, say $t_1$. However, the analysis for the static case, given in \cref{sec_CE} and \cref{sec_uni}, is written in a very general manner and it is directly applicable to the time dependent domain case of \cref{sec_timedep}. Hence, depending on the context, our proof can be adapted to the required setting. For instance, in the most general case the normal is given by 
\[\mc{N} = \sum_{i=1}^m \nu^{t_i} \p_{t_i} + \sum_{j=1}^n \nu^{x_j} \p_{x_j}.\]
In the time static case, we get $\mc{N} = \sum_{j=1}^n \nu^{x_j} \p_{x_j}$, which is the usual Euclidean setting. Thus, for this case $\mc{N}z = \nu^x \cdot z$. For the time dependent domains considered in \cref{sec_timedep}, we have the form
\[\mc{N} = \nu^{t_1} \p_{t_1} + \sum_{j=1}^n \nu^{x_j} \p_{x_j}.\]
We show the action of $\mc{N}$ in \cref{eq_N_static} for static domains and \cref{eq_N_moving} for time dependent domains.

\subsection{Features}
We obtain uniqueness of solution for ultrahyperbolic equations, in $\Rmn$ for any $m,n \in \N$, with the following features.
\begin{enumerate}

\item This is the first result that addresses uniqueness of ultrahyperbolic equation with general zeroth and first order terms. That is $V$ and $\X$ are dependent on both space and time. The result that comes the closest is that of \cite{MR348288}, which considers the equation without any first order term. Adding a first order term that is dependent on both space and time introduces additional challenges and one has to be careful with the analysis to manage the contribution of such a term. For details, see the proof following equation \eqref{eq_uni_pf_3}.

\item We obtain a novel interior Carleman estimate for ultrahyperbolic equations, with general lower order terms.
 
\item We address uniqueness from the boundary data as well as when data is given in an interior subset of the domain. To the author's knowledge, this is the first result that addresses uniqueness from interior data for ultrahyperbolic equations.

\item This is the first work that studies ultrahyperbolic equation on domains that may change with respect to time. We assume that the change can occur along any one temporal coordinate, say, the first time component $t_1$. 

\item Finally, our results are based on coordinate system centred around any arbitrary point $p:=(t(p),x(p)) \in \Rmn$. Hence, depending on where the point $p$ lies with respect to the domain $\U$, we get different uniqueness results. The point $p$ is also the reference point about which we apply the Carleman estimate. Moreover, this is the first result that obtains uniqueness for such equations when the reference point is outside the domain.

\end{enumerate}

\begin{figure}
    \centering
\includegraphics[height=5cm]{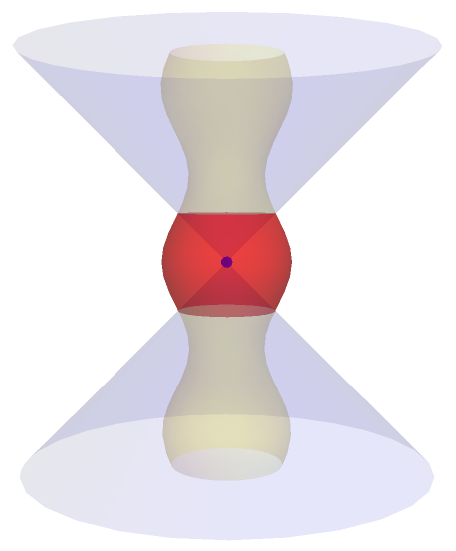} \hspace{5mm}
\includegraphics[height=5cm]{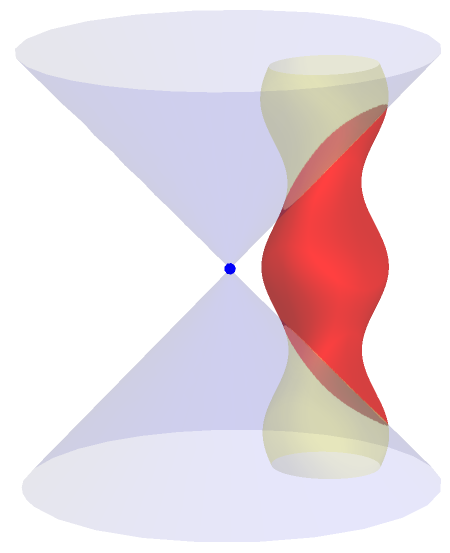} \hspace{5mm}
\includegraphics[height=5cm]{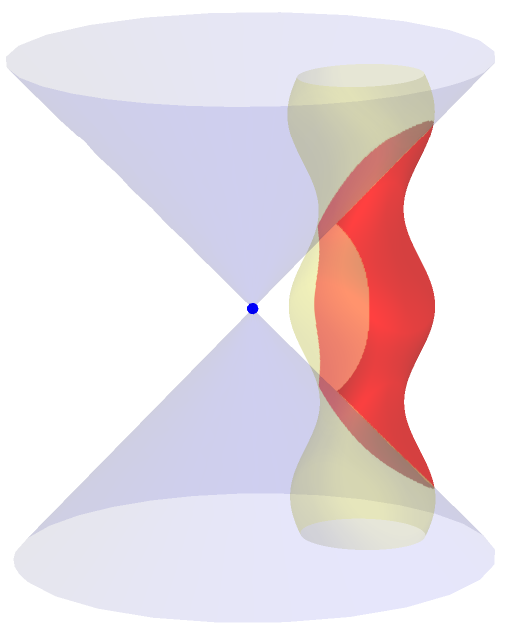}

\caption{This set of figures is for the case when the domain $\U$ is time dependent and in the setting of $(t,x) \in \R \times \R^2$. In the first image, $p \in \U$ and the red portion denotes $\p\U \cap \D_p$ as well as $\Gamma_p$. The second image is when $p \notin \bar{\U}$ and here the red part denotes $\p\U \cap \D_p$. The third image is when $p \notin \bar{\U}$ and the red part denotes $\Gamma_p$.}
    \label{fig2}

\end{figure}

The main benefit of our method is that we use Carleman estimates, which have the advantage of being applicable to a wide class of PDEs. In particular, they can be applied when the PDE has space-time dependent lower order terms. This allows us to consider the uniqueness problem for the ultrahyperbolic equation with general lower order terms. Moreover, we consider some appropriately defined cut-off functions that allow us to estimate the boundary term in the Carleman estimate of \cref{thm.carl_bdry}, by suitable terms restricted to the interior of the domain. This allows us to obtain a novel interior Carleman estimate and, subsequently, uniqueness from interior data as well.

Finally, our analysis is based on estimates that are derived using pseudo-Riemannian geometry techniques, which allows us to treat the space and time components at the same level. Hence, we can talk about domains that can also change with respect to time. However, we only address the case when the domain may change along $t_1$. This is because we treat the time dependent case by making a suitable coordinate transformation, and it is not clear what the transformation would be when the domain changes with respect to more than one temporal coordinate. Furthermore, according to the author's understanding, it is difficult to grasp what a domain changing along multiple time components looks like. This limits the lines of attack for the corresponding problem.

\subsection{Outline}
In Section \ref{sec_setting}, we present the geometric setting for the problem and provide some preliminary definitions.
In Section \ref{sec_CE}, we present the boundary Carleman estimate from \cite{MR4314050} and prove a novel interior Carleman estimate for ultrahyperbolic equations.
In Section \ref{sec_uni}, we prove the uniqueness results for the ultrahyperbolic system \eqref{eq_main_z_A}.
In \cref{sec_timedep}, we present the case of the time dependent domain and address how we can generalise the previous results to this setting.
In \cref{sec_conc} we mention some possible future directions of this work.

\section{Geometry of the domain} \label{sec_setting}
We begin by presenting the geometric background for the considered problem. One can see the following discussion to be appropriate generalisations of standard Minkowski geometry ideas to multiple time dimensions.
\begin{definition}
Let $m, n \in \mathbb{N}$ be fixed. Let $t := (t_1, \ldots,t_m)$ and $x:=(x_1,\ldots,x_n)$ denote the Cartesian coordinates on $\Rmn$. Consider the pseudo-Riemannian metric given by 
\begin{equation} \label{eq_g_metric}
g := -\rd t_1^2 -\rd t_2^2 - \cdots -\rd t_m^2 + \rd x_1^2 + \rd x_2^2 + \cdots + \rd x_n^2.
\end{equation}
Let $\tau := |t|$ and $r:= |x|$ denote the temporal and spatial radial functions, respectively. That is
\[ \tau = \sqrt{t_1^2 + \cdots + t_m^2}, \qquad r = \sqrt{x_1^2 + \cdots + x_n^2}.\]
Then the null coordinates $(u,v)$ are given by
\begin{equation} \label{eq_uv_def}
u := \frac{1}{2} (\tau-r), \qquad v:= \frac{1}{2} (\tau+r).
\end{equation}
Finally, we define the following function
\begin{equation} \label{eq_f_def}
f:= -uv = \frac{1}{4} (r^2-\tau^2) = \frac{1}{4} (|x|^2-|t|^2).
\end{equation}
\end{definition}
The function $f$ is essentially the distance function analogue, of the usual Euclidean distance, in the current geometry. The scaling factor of $\frac{1}{4}$ is present because it is more convenient to define $f=-uv$, as we will use this definition more than the expression of $f$ in terms of $(r,\tau)$ or $(x,t)$. 

Let $\mb{S}^{m-1}$ and $\mb{S}^{n-1}$ denote the unit sphere in $\R^m$ and $\R^n$, respectively.  On $\{ \tau \neq 0, r \neq 0\}$, we write $(\tau,r,\omega_x,\omega_t)$ to denote the standard Cartesian coordinates and $(u,v,\omega_x,\omega_t)$ to denote the null coordinates. Here $\omega_x$ is the spatial angular component taking values in $\mb{S}^{n-1}$ and $\omega_t$ is the temporal angular component with values in $\mb{S}^{m-1}$. Then we can write the metric given in \eqref{eq_g_metric} as
\[g = -d\tau^2 + dr^2 + r^2 \mathring{\gamma}_{\mathbb{S}^{n-1}} - \tau^2 \mathring{\gamma}_{\mathbb{S}^{m-1}} = -4dudv + r^2 \mathring{\gamma}_{\mathbb{S}^{n-1}} - \tau^2 \mathring{\gamma}_{\mathbb{S}^{m-1}},\]
where \( \mathring{\gamma}_{\mathbb{S}^{n-1}} \) and \( \mathring{\gamma}_{\mathbb{S}^{m-1}} \) denote the unit round metric on \( \mathbb{S}^{n-1} \) and \(  \mathbb{S}^{m-1} \), respectively.
\begin{definition}
Define the region $\D \subset \Rmn$ as $\D := \{ (t,x) \in \Rmn : f(t,x) > 0 \}.$
\end{definition}
Note that the boundary of $\D$, denoted by $\p\D$, is just the region $\{f=0\}$. Furthermore, using \eqref{eq_uv_def} and \eqref{eq_f_def} shows that we have the following estimates on $\D$
\begin{equation} \label{eq_uvf_est}
0<-u<r, \quad 0<v<r, \quad 0<f<r^2.
\end{equation}
The above estimates will be used repeatedly throughout the article.
To present our results in the most general form, we also define the above notions with respect to any arbitrary point $p \in \Rmn$. 

\begin{definition}
Fix a point $p:=(t(p),x(p)) \in \Rmn$. Define
\[t_p := t - t(p),\ \ x_p := x - x(p),\ \ r_p := |x_p|,\ \ \tau_p:=|t_p|,\ \ u_p := \frac{1}{2} (t_p - r_p),\ \ v_p := \frac{1}{2} (t_p + r_p).\]
We also define $ f_p := - u_p v_p$ and the region $\D_p := \{ f_p > 0 \}$. 
\end{definition}
With the above definition, we can write the metric $g$ as
\[ g = -dt_p^2 + dr_p^2 + r_p^2 \mathring{\gamma}_{\mathbb{S}^{n-1}} - \tau_p^2 \mathring{\gamma}_{\mathbb{S}^{m-1}} = -4 du_p dv_p + r_p^2 \mathring{\gamma}_{\mathbb{S}^{n-1}} - \tau_p^2 \mathring{\gamma}_{\mathbb{S}^{m-1}}. \]

We now mention some convention that will be used in the article. Lower case Greek letters \( (\alpha, \beta, \ldots)\), ranging from 1 to \(m+n\), denote space-time components in \(\Rmn\). Lower case Latin letters \( (a,b,\ldots)\), ranging from 1 to \(n-1\), denote spatial angular components corresponding to \(\omega_x \in \mathbb{S}^{n-1} \) in the above mentioned coordinate systems.
Lower case Latin letters \( (A,B,\ldots)\), ranging from 1 to \(n-1\), denote temporal angular components corresponding to \(\omega_t \in \mathbb{S}^{m-1} \) in the above mentioned coordinate systems. The symbol \( \nabla \) denotes the Levi-Civita connection with respect to \(g\). The symbols \( \nasla \) and  \( \tilde\nabla \) denote derivatives in the spatial and temporal angular components with respect to \(g\), respectively.
The ultrahyperbolic operator with respect to  \(g\) is given by \( \square := g^{\alpha\beta}  \nabla_{\alpha\beta} \).

\section{Carleman estimate} \label{sec_CE}
We first present a boundary Carleman estimate for the ultrahyperbolic operator (\cite{MR4314050}) that is also applicable on time-dependent domains. 

\begin{theorem}[Boundary Carleman Estimate] \label{thm.carl_bdry}
Let  \(\mf{U} \subset \R^{m+n}\) be such that for some \( R>0 \)
\[\mf{U} \cap \mf{D}_p \subseteq \{ r_p<R \}.\]
Also, assume that the boundary of \(\mf{U}\), denoted by \( \partial \mf{U} \), is smooth and timelike. Let \( \varepsilon, a, b >0 \) be constants such that:
\begin{equation}\label{eq.carleman_choices}
a \geqslant (m+n)^2, \qquad \varepsilon \ll_{m,n} b \ll R^{-1}.
\end{equation}
Then, there exists \( C >0 \) such that for any \(z \in \mc{C}^2({\mf{U}})\cap \mc{C}^1(\bar{\mf{U}}) \) with $z|_{\partial\mf{U} \cap \mf{D}_p} = 0$, we have 
\begin{align}\label{eq.carleman_est1}
& C\varepsilon \int_{\mf{U}\cap \mf{D}_p}\zeta_{a,b;\varepsilon}^p r_p^{-1}(|u_p \partial_{u_p} z|^2 + |v_p \partial_{v_p} z|^2 + f_p g^{ab}\slashed\nabla_a^p z \slashed\nabla_b^p z - f_p g^{CD} \tilde{\nabla}_C^p z \tilde{\nabla}_D^p z ) \\
&  + Cba^2\int_{\mf{U}\cap \mf{D}_p}\zeta_{a,b;\varepsilon}^p f^{-\frac{1}{2}} z^2 \leqslant \frac{1}{a}\int_{\mf{U}\cap \mf{D}_p} \zeta_{a,b;\varepsilon}^p f_p |\square z|^2 + C' \int_{\partial\mf{U}\cap \mf{D}_p} \zeta_{a,b;\varepsilon}^p [( 1 - \varepsilon r_p ) \mc{N} f_p + \varepsilon f_p \mc{N} r_p ] |\mathcal{N} z|^2, \notag
\end{align}
where \( \zeta_{a,b;\varepsilon}^p \) is the Carleman weight defined as
\begin{equation}
\label{eq.carleman_weight} \zeta_{ a, b; \varepsilon }^p := \left\{ \frac{ f_p }{ ( 1 + \varepsilon u_p ) ( 1 - \varepsilon v_p ) } \cdot \exp \left[ \frac{ 2 b f_p^\frac{1}{2} }{ ( 1 + \varepsilon u_p )^\frac{1}{2} ( 1 - \varepsilon v_p )^\frac{1}{2} } \right] \right\}^{2a} ,
\end{equation}
and \(\mc{N}\) is the outer-pointing unit normal of \(\mf{U}\) (with respect to \(g\)).
\end{theorem}
The above estimate was derived in \cite{MR4314050}, and was used to solve a controllability problem for wave equations. In particular, to overcome a regularity issue encountered while showing controllability, it was essential to obtain a Carleman estimate for wave-type operators in $(t,x) \in \R^2 \times \R^n$. This proof was generalised to get a $(t,x) \in \R^m \times \R^n$ Carleman estimate (\cref{thm.carl_bdry}) purely as a theoretical concept and later the $m=2$ case was used to prove controllability. However, since we are dealing with ultrahyperbolic equations, we wish to use the general $\R^m \times \R^n$ Carleman estimate. Indeed, we will use it to prove uniqueness of solution of \eqref{eq_main_z_A} from boundary data, in \cref{thm_uc_bdry_Dp}. When we wish to show uniqueness from the interior data, we need to use an interior Carleman estimate, which is our next result \cref{thm_carl_int}. The key step in this direction is to estimate the boundary integral term present on the right hand side of \eqref{eq.carleman_est1} by terms restricted to the interior of the domain.
\begin{remark}
We have written the analysis in Sections \ref{sec_setting}-\ref{sec_uni} in a way that is also applicable to the time dependent domains described in \cref{sec_timedep}. This is done to avoid repeating the statements and proofs. However, one can simplify the expressions in Sections \ref{sec_setting}-\ref{sec_uni} by using the fact that we have a time static domain. For instance, the normal $\mc{N}$ actually reduces to $\displaystyle \sum_{j=1}^n \nu^{x_j} \p_{x_j}$. Then, we see that 
\begin{equation} \label{eq_N_static}
\mc{N} f_p = \frac{1}{4} \mc{N} (|x_p|^2 - |t_p|^2) = \frac{1}{4}\sum_{j=1}^n \nu^{x_j} \p_{x_j} (|x_p|^2 - |t_p|^2) = \frac{1}{4}\sum_{j=1}^n \nu^{x_j} \p_{x_j} |x_p|^2 = \frac{1}{2} \nu^x \cdot x_p.
\end{equation}
Contrast this with the computations in \cref{eq_N_moving} for the time dependent case.
\end{remark}

\begin{definition} 
Let $\varepsilon > 0$ be fixed and define $\Gamma_p^\varepsilon$, a subset of $ \partial \mf{U} \cap \mf{D}_p$ as follows
\begin{equation} \label{eq_def_Gamma_MB}
\Gamma_p^\varepsilon  := \partial \mf{U} \cap \mf{D}_p \cap \{ ( 1 - \varepsilon r_p ) \mc{N} f_p + \varepsilon f_p \mc{N} r_p > 0 \} .
\end{equation}
For any $y \in \R^n$ and $\sigma > 0 $, we define the $\sigma$-neighbourhood of $y$ in $\R^n$ as follows
\[\mc{O}_\sigma(y) := \{ y_1 \in \R^n : \|y_1-y\|_{\R^n} < \sigma\} \subset \R^n.\]
Then, we define the $\sigma$-neighbourhood of the set $\Gamma_p^\varepsilon$ as 
\[\mc{O}_\sigma(\Gamma_p^\varepsilon) := \bigcup_{(t,y) \in \Gamma_p^\varepsilon} \Big(  \{ t \} \times \mc{O}_\sigma(y) \Big) \subset \R^{m+n}.\]
Finally, we define the interior region $W_p^\varepsilon$ as 
\begin{equation} \label{eq_def_W_pe}
W_p^\varepsilon := \mc{O}_\sigma (\Gamma_p^\varepsilon) \cap (\mf{U} \cap \mf{D}_p).
\end{equation}
\end{definition}

\begin{theorem}[Interior Carleman Estimate] \label{thm_carl_int}
Let the hypothesis of Theorem \ref{thm.carl_bdry} be satisfied. Additionally, also assume that $a \gg R$.
Then, there exists \( C >0 \) such that for any \(z \in \mc{C}^2({\mf{U}})\cap \mc{C}^1(\bar{\mf{U}}) \) with $z|_{\partial\mf{U} \cap \mf{D}_p} = 0$, we have 
\begin{align}\label{eq.carleman_est}
 & C \varepsilon \int_{\mf{U}\cap \mf{D}_p} \zeta_{a,b;\varepsilon}^p r_p^{-1}(|u_p \partial_{u_p} z|^2 + |v_p \partial_{v_p} z|^2 + f_p g^{ab}\slashed\nabla_a^p z \slashed\nabla_b^p z - f_p g^{CD} \tilde{\nabla}_C^p z \tilde{\nabla}_D^p z ) \\
&  + Cba^2\int_{\mf{U}\cap \mf{D}_p}\zeta_{a,b;\varepsilon}^p f^{-\frac{1}{2}} z^2 \leqslant \frac{1}{a}\int_{\mf{U}\cap \mf{D}_p} \zeta_{a,b;\varepsilon}^p f_p |\square z|^2 + a R^2 \int_{W_p^\varepsilon} \zeta_{a,b;\varepsilon}^p f_p^{-1} |\nb_t z|^2 + a^4 R^4 \int_{W_p^\varepsilon} \zeta_{a,b;\varepsilon}^p f_p^{-3} |z|^2. \notag
\end{align}
\end{theorem}

\begin{proof}
Since the assumptions of Theorem \ref{thm_carl_int} hold, this implies that the assumptions of Theorem \ref{thm.carl_bdry} and \eqref{eq.carleman_est1} hold. Then to complete the proof we only need to show that the boundary integral term in \eqref{eq.carleman_est1} can be bounded from above by suitable terms restricted to $W_p^\varepsilon$. 
First, using \eqref{eq_N_static}, we see that
\begin{align*}
\mc{N} f_p = \frac{1}{2} \nu^x \cdot x_p, \qquad \mc{N} r_p = \frac{1}{2 r_p} \mc{N} (r_p^2) = \frac{1}{2 r_p} \mc{N} (|x_p|^2) = \frac{1}{r_p} \nu^x \cdot x_p,
\end{align*}
which implies that
\begin{align*}
( 1 - \varepsilon r_p ) & \mc{N} f_p + \varepsilon f_p \mc{N} r_p = ( 1 - \varepsilon r_p ) \frac{1}{2} \nu^x \cdot x_p + \varepsilon f_p \frac{1}{r_p} \nu^x \cdot x_p = \left( 1 - \frac{\varepsilon r_p }{2} + \frac{\varepsilon f_p}{r_p} \right) \nu^x \cdot x_p\\
& \leqslant \left( 1 - \frac{\varepsilon r_p }{2} + \frac{\varepsilon r_p^2}{r_p} \right) \nu^x \cdot x_p \leqslant \left( 1 + \frac{\varepsilon r_p }{2} \right) \nu^x \cdot x_p \lesssim R,
\end{align*}
where we have used \eqref{eq_uvf_est} and that $\varepsilon \ll R^{-1}$.
Then using \eqref{eq_def_Gamma_MB} and the above shows that
\begin{align*}
C' \int_{\partial\mf{U}\cap \mf{D}_p} \zeta_{a,b;\varepsilon}^p [( 1 - \varepsilon r_p ) \mc{N} f_p + \varepsilon f_p \mc{N} r_p ] |\mathcal{N} z|^2 & \leqslant C' \int_{\Gamma_p^\varepsilon} \zeta_{a,b;\varepsilon}^p [( 1 - \varepsilon r_p ) \mc{N} f_p + \varepsilon f_p \mc{N} r_p] |\mathcal{N} z|^2 \\
& \leqslant C' R \int_{\Gamma_p^\varepsilon} \zeta_{a,b;\varepsilon}^p |\mathcal{N} z|^2.
\end{align*}
Thus our goal is to estimate the term on the right hand side by an interior term, which will be done by defining a suitable cut-off function.
Next, we take a vector field \( h \in C^1(\bar{\U};\Rmn) \) such that \(h = \mc{N} \) on \(\partial \U \cap \D_p\) and define the cut-off function \( \rho\in C^2(\bar{\U};[0,1]) \) as follows
\begin{equation} \label{eq:3.1_MB}
\rho(t,x)
= \begin{cases}
1 ,\hspace{1cm} (t,x) \in \mathcal{O}_{\sigma/3}(\Gamma_p^\varepsilon) \cap \U ,\\
0 , \hspace{1cm} (t,x) \in \U \setminus \mathcal{O}_{\sigma/2}(\Gamma_p^\varepsilon) .
\end{cases}
\end{equation}
Now multiplying $\square z$ with the multiplier $\rho \zeta_{a,b;\varepsilon}^p h z$, integrating on $\U\cap \D_p$, and using integration by parts, gives us
\begin{align} \label{eq_g10_MB}
\int_{\U\cap \D_p} \square z \rho \zeta_{a,b;\varepsilon}^p h z = \int_{\U\cap \D_p} \nb^\alpha \nb_\alpha z \rho \zeta_{a,b;\varepsilon}^p h z = - \int_{\U\cap \D_p}\nabla_\alpha z \nabla^\alpha (\rho \zeta_{a,b;\varepsilon}^p h z) + \int_{\partial\U\cap \D_p}\mathcal{N} z \rho \zeta_{a,b;\varepsilon}^p h z ,
\end{align}
where we get no contribution from the boundary \( \U \cap \partial\D_p \), because $f_p|_{\partial \D_p} =0$ which implies that \( \zeta_{a,b;\varepsilon}^p|_{\partial \D_p} = 0 \). 
For the first term on the right hand side, applying integration by parts and using $z|_{\partial \U \cap \D_p} =0$, shows
\begin{align*}
& \int_{\U\cap \D_p} \nabla_\alpha z \nabla^\alpha ( \rho \zeta_{a,b;\varepsilon}^p h z) \\
& =  \int_{\U \cap \D_p} \rho \zeta_{a,b;\varepsilon}^p \nabla_\alpha z \nabla^\alpha(h^\beta \nabla_\beta z) + \int_{\U \cap \D_p} \nabla_\alpha z \nabla^\alpha(\rho \zeta_{a,b;\varepsilon}^p) \cdot h z \\
& = \int_{\U \cap \D_p} \rho \zeta_{a,b;\varepsilon}^p \nabla_\alpha z \nabla^\alpha h^\beta \nabla_\beta z + \int_{\U \cap \D_p} \rho \zeta_{a,b;\varepsilon}^p \nabla_\alpha z \cdot h^\beta {\nabla^\alpha}_\beta z + \int_{\U \cap \D_p} \nabla_\alpha z \nabla^\alpha(\rho \zeta_{a,b;\varepsilon}^p) \cdot h z \\
& = \int_{\U \cap \D_p} \rho \zeta_{a,b;\varepsilon}^p \nabla^\alpha h^\beta \nabla_\alpha z \nabla_\beta z + \frac{1}{2} \int_{\U \cap \D_p} \rho \zeta_{a,b;\varepsilon}^p h^\beta \nabla_\beta (\nabla_\alpha z \nabla^\alpha z) + \int_{\U \cap \D_p} \nabla_\alpha z \nabla^\alpha(\rho \zeta_{a,b;\varepsilon}^p) \cdot h z \\
& =  \int_{\U \cap \D_p} \rho \zeta_{a,b;\varepsilon}^p \nabla^\alpha h^\beta \nabla_\alpha z \nabla_\beta z - \frac{1}{2} \int_{\U \cap \D_p} \nabla_\beta (\rho \zeta_{a,b;\varepsilon}^p h^\beta ) \nabla_\alpha z \nabla^\alpha z + \frac{1}{2} \int_{\partial \U \cap \D_p} \rho \zeta_{a,b;\varepsilon}^p |\mathcal{N} z|^2 \\
& \qquad  + \int_{\U \cap \D_p} \nabla_\alpha z \nabla^\alpha(\rho \zeta_{a,b;\varepsilon}^p) \cdot h z .
\end{align*}
Substituting the above expression in \eqref{eq_g10_MB} and multiplying with the coefficient $R$, we get 
\begin{align*}
\frac{R}{2} \int_{\partial\U\cap \D_p}\rho\zeta_{a,b;\varepsilon}^p |\mathcal{N} z|^2 & = R \int_{\U\cap \D_p}\square z \rho\zeta_{a,b;\varepsilon}^p h z + R \int_{\U \cap \D_p} \rho \zeta_{a,b;\varepsilon}^p \nabla^\alpha h^\beta \nabla_\alpha z \nabla_\beta z \numberthis \label{eq_g15_MB}\\
& \qquad - \frac{R}{2} \int_{\U \cap \D_p} \nabla_\beta (\rho \zeta_{a,b;\varepsilon}^p h^\beta ) \nabla_\alpha z \nabla^\alpha z + R \int_{\U \cap \D_p} \nabla_\alpha z \nabla^\alpha(\rho \zeta_{a,b;\varepsilon}^p) hz . 
\end{align*}
For the term on the left hand side since $\Gamma_p^\varepsilon \subset \p\U \cap \D_p$, we can reduce the corresponding integral region. Indeed, using this fact along with \eqref{eq:3.1_MB}, we get
\begin{equation} \label{eq_g151_MB}
R \int_{\partial\U \cap \D_p} \rho \zeta_{a,b;\varepsilon}^p |\mathcal{N} z|^2 \geqslant R \int_{\Gamma_p^\varepsilon} \rho \zeta_{a,b;\varepsilon}^p |\mathcal{N} z|^2 \geqslant R \int_{\Gamma_p^\varepsilon} \zeta_{a,b;\varepsilon}^p |\mathcal{N} z|^2.
\end{equation}
Next, the first term on the right hand side of \eqref{eq_g15_MB} can be estimated as follows
\begin{equation} \label{eq_g155_MB}
R \int_{\U\cap \D_p}\square z \rho\zeta_{a,b;\varepsilon}^p h z \lesssim \frac{1}{a} \int_{\U \cap \D_p} \rho \zeta_{a,b;\varepsilon}^p f_p |\square z|^2 + a R^2 \int_{\U \cap \D_p} \rho \zeta_{a,b;\varepsilon}^p f_p^{-1}| h z|^2.
\end{equation}
For estimating the terms containing derivatives of $\zeta_{a,b;\varepsilon}^p$, we first prove the following claim.

\noindent \emph{Claim}: The following is satisfied
\begin{equation}\label{prf.der_zeta}
|\nabla^\alpha\zeta_{a,b;\varepsilon}^p| \lesssim a R \zeta_{a,b;\varepsilon}^p f_p^{-1}.
\end{equation}
\emph{Proof of claim}:
Using the definitions of \(u,v,f\), we get 
\[ (1 + \varepsilon u_p) (1-\varepsilon v_p) = 1 + \varepsilon (u_p-v_p) - \varepsilon u_p v_p = 1- \varepsilon r_p + \varepsilon^2 f_p . \]
Then, using the definition of $\zeta_{a,b;\varepsilon}^p$ from \eqref{eq.carleman_weight} shows
\begin{align*} 
\nabla^\alpha \zeta_{a,b;\varepsilon}^p & = \nabla^\alpha\Bigg\{\frac{f_p}{(1- \varepsilon r_p + \varepsilon^2 f_p )}\cdot\text{exp} \Bigg[\frac{2bf_p^{\frac{1}{2}}}{(1- \varepsilon r_p + \varepsilon^2 f_p )^\frac{1}{2}}\Bigg] \Bigg\}^{2a}\\
 & = 2a \Bigg\{\frac{f_p}{(1- \varepsilon r_p + \varepsilon^2 f_p )}\cdot\text{exp} \bigg[\frac{2bf_p^{\frac{1}{2}}}{(1- \varepsilon r_p + \varepsilon^2 f_p )^\frac{1}{2}}\bigg] \Bigg\}^{2a-1} \\
& \qquad \cdot \nb^\alpha \Bigg(\frac{f_p}{(1- \varepsilon r_p + \varepsilon^2 f_p )}\cdot\text{exp} \bigg[\frac{2bf_p^{\frac{1}{2}}}{(1- \varepsilon r_p + \varepsilon^2 f_p )^\frac{1}{2}}\bigg] \Bigg) \\
& = 2 a \zeta_{a,b;\varepsilon}^p \Bigg\{\frac{f_p}{(1- \varepsilon r_p + \varepsilon^2 f_p )}\cdot\text{exp} \bigg[\frac{2bf_p^{\frac{1}{2}}}{(1- \varepsilon r_p + \varepsilon^2 f_p )^\frac{1}{2}}\bigg] \Bigg\}^{-1} \Bigg\{ \nb^\alpha \bigg( \frac{f_p}{(1- \varepsilon r_p + \varepsilon^2 f_p )} \bigg) \\
& \qquad \cdot \text{exp} \bigg[\frac{2bf_p^{\frac{1}{2}}}{(1- \varepsilon r_p + \varepsilon^2 f_p )^\frac{1}{2}}\bigg] + \frac{f_p}{(1- \varepsilon r_p + \varepsilon^2 f_p )} \cdot \nb^\alpha \text{exp} \bigg[\frac{2bf_p^{\frac{1}{2}}}{(1- \varepsilon r_p + \varepsilon^2 f_p )^\frac{1}{2}}\bigg] \Bigg\} \\
& = 2a \zeta_{a,b;\varepsilon}^p \Bigg\{ \frac{(1- \varepsilon r_p + \varepsilon^2 f_p )}{f_p} \nabla^\alpha\bigg( \frac{f_p}{(1- \varepsilon r_p + \varepsilon^2 f_p )} \bigg) + \nabla^\alpha \Bigg[\frac{2bf_p^{\frac{1}{2}}}{(1- \varepsilon r_p + \varepsilon^2 f_p )^\frac{1}{2}}\Bigg]\Bigg\}. \nt \label{eq.zeta_der_0}  
\end{align*}
Next, consider the following computations
\begin{align*}
\nabla^\alpha\bigg( \frac{f_p}{(1- \varepsilon r_p + \varepsilon^2 f_p )} \bigg) & = \frac{\nabla^\alpha f_p}{(1- \varepsilon r_p + \varepsilon^2 f_p )} - \frac{f_p \nabla^\alpha (1- \varepsilon r_p + \varepsilon^2 f_p )}{(1- \varepsilon r_p + \varepsilon^2 f_p )^2} \\
& = \frac{\nabla^\alpha f_p}{(1- \varepsilon r_p + \varepsilon^2 f_p )} - \frac{f_p (-\varepsilon \nabla^\alpha r_p + \varepsilon^2 \nabla^\alpha f_p)}{(1- \varepsilon r_p + \varepsilon^2 f_p )^2}
\end{align*}
Next, by \eqref{eq_g_metric} we have that \( | \nabla^\alpha r | \leqslant 1 \), and by we get \eqref{eq_uvf_est} \( |\nabla^\alpha f_p| \lesssim R \). The second term on the right hand side of \eqref{eq.zeta_der_0} is estimated in a similar manner. Finally, using \eqref{eq.carleman_choices} gives
\begin{align*}
|\nabla^\alpha \zeta_{a,b;\varepsilon}^p| \lesssim 2 a \zeta_{a,b;\varepsilon} \left( \frac{(1- \varepsilon r_p + \varepsilon^2 f_p ) R}{f_p}  + \frac{ b R }{ f_p^\frac{1}{2} } \right) \lesssim a R & \zeta_{a,b;\varepsilon} ( f_p^{-1} + b f_p^{-\frac{1}{2}}) \\
& \lesssim a R \zeta_{a,b;\varepsilon} f_p^{-1} ( 1 + b f_p^{\frac{1}{2}}) \lesssim a R  \zeta_{a,b;\varepsilon}^p f_p^{-1},
\end{align*}
which completes the proof of the claim.

We will now use the above claim to estimate the third and fourth terms on the right hand side of \eqref{eq_g15_MB}. First, we have
\begin{align*} 
|\nabla_\beta (\rho \zeta_{a,b;\varepsilon}^p h^\beta )| & \leqslant |\nb_\beta \rho \cdot \zeta_{a,b;\varepsilon}^p h^\beta + \rho \nb_\beta \zeta_{a,b;\varepsilon}^p h^\beta + \rho \zeta_{a,b;\varepsilon}^p \nb_\beta h^\beta| \\
& \lesssim |\nb_\beta \rho h^\beta \zeta_{a,b;\varepsilon}^p| + a R \zeta_{a,b;\varepsilon}^p f_p^{-1} \rho + \rho \zeta_{a,b;\varepsilon}^p |\nb_\beta h^\beta| \\
& \lesssim |\nb_\beta \rho h^\beta \zeta_{a,b;\varepsilon}^p| \cdot f_p f_p^{-1} + a R \zeta_{a,b;\varepsilon}^p f_p^{-1} \rho + \rho \zeta_{a,b;\varepsilon}^p |\nb_\beta h^\beta| \cdot f_p f_p^{-1} \\
& \lesssim a R (|h^\beta \nb_\beta \rho| + \rho) \zeta_{a,b;\varepsilon}^p f_p^{-1}, \nt \label{eq_g15666}
\end{align*}
where in the last step we have used the fact that $f_p < r^2 < R^2 < a R$. A similar computation also shows that $|\nabla^\alpha (\rho \zeta_{a,b;\varepsilon}^p)| < a R (|\nb^\alpha \rho| + \rho) \zeta_{a,b;\varepsilon}^p f_p^{-1}$. Substituting this fact along with \eqref{eq_g151_MB} and \eqref{eq_g155_MB}, in \eqref{eq_g15_MB} shows that
\begin{align*}
R & \int_{\Gamma_p^\varepsilon} \zeta_{a,b;\varepsilon}^p |\mathcal{N} z|^2 \nt \label{eq_g157} \\
& \lesssim \frac{1}{a} \int_{\U\cap \D_p} \rho \zeta_{a,b;\varepsilon}^p f_p |\square z|^2 + a R^2 \int_{\U \cap \D_p} \rho \zeta_{a,b;\varepsilon}^p f_p^{-1}| h z|^2 + R \int_{\U \cap \D_p} \rho \zeta_{a,b;\varepsilon}^p \nabla^\alpha h^\beta \nabla_\alpha z \nabla_\beta z\\
& \qquad  + a R^2 \int_{\U \cap \D_p} (|h^\beta \nb_\beta \rho| + \rho) \zeta_{a,b;\varepsilon}^p f^{-1} \nabla_\alpha z \nabla^\alpha z + aR^2 \int_{\U \cap \D_p} (|\nb^\alpha \rho| + \rho) \zeta_{a,b;\varepsilon}^p f^{-1} | \nb_\alpha z hz|. 
\end{align*}
Using the properties of the cut-off function $\rho$ from \eqref{eq:3.1_MB}, we have
\begin{equation} \label{eq:g2_MB} 
R \int_{ \Gamma_p^\varepsilon } \zeta_{a,b;\varepsilon}^p |\mathcal{N} z|^2 \lesssim \frac{1}{a} \int_{\U \cap \D_p} \zeta_{a,b;\varepsilon}^p f_p |\square z|^2 + a R^2 \int_{\mathcal{O}_{\sigma/2}(\Gamma_p^\varepsilon) \cap (\U \cap \D_p)} \zeta_{a,b;\varepsilon}^p f_p^{-1} ( |\nabla_{t,x} z|^2).
\end{equation}
Our next goal is to bound the $|\nb_x z|$ term on the right hand side from above by data restricted to $W_p^\varepsilon$. 
Define the cut-off function \( \rho_1\in C^2 ( \Bar{\U};[0,1] ) \) as
\begin{equation} \label{eq:rho_1_MB}
\rho_1(t,x) = \begin{cases}
1 , \qquad (t,x) \in \mathcal{O}_{\sigma/2}(\Gamma_p^\varepsilon) \cap \U \cap \D_p , \\
0 , \qquad (t,x) \in (\U \cap \D_p) \setminus W_p^\varepsilon.
\end{cases} 
\end{equation}
Furthermore, also define the function $\eta(t,x):=\rho_1^2 \zeta_{a,b;\varepsilon}^p f_p^{-1}$. Next, we use integration by parts to get the following
\begin{align*}
\int_{\U \cap \D} \eta z \square z = - \int_{\U \cap \D} z \nb^\alpha \eta \nb_\alpha z - \int_{\U \cap \D} \eta \nb^\alpha z \nb_\alpha z,
\end{align*}
where we do not see any boundary terms because $z$ vanishes on $\p\U \cap \D$ and $\zeta_{a,b;\varepsilon}^p$ vanishes on $\U \cap \p\D_p$. The above equation implies that
\begin{equation} \label{eq_g110}
\int_{\U \cap \D_p}\eta |\nabla_x z|^2 = - \int_{\U \cap \D_p}\eta z \square z + \int_{\U\cap \D_p} z \nb_t z \cdot \nb_t \eta - \int_{\U\cap \D_p} z \nabla_x z \cdot \nabla_x \eta + \int_{\U\cap \D_p} \eta |\partial_t z|^2.
\end{equation}
Before proceeding further, we give estimates on the derivatives of $\eta$, which can be calculated using similar ideas as \eqref{eq_g15666}:
\[|\nabla_x \eta| \lesssim \rho_1 \zeta_{a,b;\varepsilon}^p \left( f_p^{-1} |\nabla_x \rho_1| + a R \rho_1 f_p^{-2} \right), \qquad|\nb_t \eta| \lesssim \rho_1 \zeta_{a,b;\varepsilon}^p \left( f_p^{-1} |\nabla_t \rho_1| + a R \rho_1 f_p^{-2} \right).\]
Substituting the above in \eqref{eq_g110}, we get
\begin{align*}
\int_{\U\cap \D_p} & \rho_1^2 \zeta_{a,b;\varepsilon}^p f_p^{-1} | \nabla_x z|^2 \\
& \lesssim \int_{\U\cap \D_p} \rho_1^2\zeta_{a,b;\varepsilon}^p f_p^{-1} | z\square z| + \int_{\U\cap \D_p}\rho_1 \zeta_{a,b;\varepsilon}^p \left[ f_p^{-1} |\nabla_x \rho_1| + a R \rho_1 f_p^{-2} \right] |z \nb_t z | \\
& \qquad  + \int_{\U\cap \D_p} \rho_1 \zeta_{a,b;\varepsilon}^p \left[ f_p^{-1} |\nabla_x \rho_1| + a R \rho_1 f_p^{-2} \right] | z \nabla_x z | + \int_{\U\cap\D_p}\rho_1^2\zeta_{a,b;\varepsilon}^p f_p^{-1}|\nb_t z|^2.
\end{align*}
Since there is a constant $aR^2$ with the $\nb_x z$ term in \eqref{eq:g2_MB} that we want to estimate, we multiply the above estimate with $aR^2$ throughout to get
\begin{align*}
aR^2 & \int_{\U\cap \D_p} \rho_1^2 \zeta_{a,b;\varepsilon}^p f_p^{-1} | \nabla_x z|^2  \numberthis \label{eq_g112_MB} \\
& \lesssim a R^2 \int_{\U\cap \D_p} \rho_1^2\zeta_{a,b;\varepsilon}^p f_p^{-1} | z \square z| + aR^2 \int_{\U\cap \D_p}\rho_1 \zeta_{a,b;\varepsilon}^p \left[ f_p^{-1} |\nabla_x \rho_1| + a R \rho_1 f_p^{-2} \right] |z \nb_t z | \\
& \qquad + aR^2 \int_{\U\cap \D_p} \rho_1 \zeta_{a,b;\varepsilon}^p \left[ f_p^{-1} |\nabla_x \rho_1| + a R \rho_1 f_p^{-2} \right] | z \nabla_x z | + aR^2 \int_{\U\cap\D_p}\rho_1^2\zeta_{a,b;\varepsilon}^p f_p^{-1}|\nb_t z|^2.
\end{align*}
For the first term on the right hand side of the above estimate, we use Cauchy-Schwarz inequality to get
\begin{equation} \label{eq_g1121}
a R^2 \int_{\U\cap \D_p} \rho_1^2\zeta_{a,b;\varepsilon}^p f_p^{-1} | z \square z| \lesssim \frac{1}{a} \int_{\U\cap \D_p} \rho_1^2 \zeta_{a,b;\varepsilon}^p f_p |\square z|^2 + a^3 R^4 \int_{\U\cap \D_p} \rho_1^2 \zeta_{a,b;\varepsilon}^p f_p^{-3} z^2.
\end{equation}
Now we will estimate the second term on the right hand side of \eqref{eq_g112_MB}:
\begin{align*}
& a R^2 \int_{\U\cap \D_p} \rho_1 \zeta_{a,b;\varepsilon}^p \left[ f_p^{-1} |\nabla_t \rho_1| + a R \rho_1 f_p^{-2} \right] | z \nabla_t z | \\
& \lesssim a R^2 \int_{\U\cap \D_p} \rho_1 \zeta_{a,b;\varepsilon}^p f_p^{-1} |\nabla_t \rho_1| | z \nabla_t z | + a^2 R^3 \int_{\U\cap \D_p} \rho_1^2 \zeta_{a,b;\varepsilon}^p f_p^{-2} | z \nabla_t z | \\
& \lesssim R^2 \int_{\U\cap \D_p} \zeta_{a,b;\varepsilon}^p f_p^{-1} (a|\nabla_t \rho_1 z |) \cdot (\rho_1 |\nabla_t z|) + \int_{\U\cap \D_p} \rho_1^2 \zeta_{a,b;\varepsilon}^p f_p^{-1} (R |\nb_t z|) \cdot (a^2 R^2 f_p^{-1}|z|)  \\
& \lesssim a^2 R^2 \int_{\U\cap \D_p} \zeta_{a,b;\varepsilon}^p f_p^{-1} |\nabla_t \rho_1|^2 z^2 + R^2 \int_{\U\cap \D_p} \zeta_{a,b;\varepsilon}^p f_p^{-1} \rho_1^2 |\nabla_t z|^2 + a^4 R^4 \int_{\U\cap \D_p} \rho_1^2 \zeta_{a,b;\varepsilon}^p f_p^{-3} z^2, \nt \label{eq_g1122}
\end{align*}
where we used the Cauchy-Schwarz inequality to conclude the last step. A similar computation shows that the third term on the right hand side of \eqref{eq_g112_MB} satisfies
\begin{align*}
& a R^2 \int_{\U\cap \D_p} \rho_1 \zeta_{a,b;\varepsilon}^p \left[ f_p^{-1} |\nabla_x \rho_1| + a R \rho_1 f_p^{-2} \right] | z \nabla_x z | \\
& \lesssim a^2 R^2 \int_{\U\cap \D_p} \zeta_{a,b;\varepsilon}^p f_p^{-1} |\nabla_x \rho_1|^2 z^2 + R^2 \int_{\U\cap \D_p} \zeta_{a,b;\varepsilon}^p f_p^{-1} \rho_1^2 |\nabla_x z|^2 + a^4 R^4 \int_{\U\cap \D_p} \rho_1^2 \zeta_{a,b;\varepsilon}^p f_p^{-3} z^2.
\end{align*}
Using the above along with \eqref{eq_g1121} and \eqref{eq_g1122}, in \eqref{eq_g112_MB}, we get
\begin{align*}
& aR^2 \int_{\U\cap \D_p} \rho_1^2 \zeta_{a,b;\varepsilon}^p f_p^{-1} | \nabla_x z|^2 \\
& \lesssim \frac{1}{a} \int_{\U\cap \D_p} \rho_1^2 \zeta_{a,b;\varepsilon}^p f_p |\square z|^2 + a^3 R^4 \int_{\U\cap \D_p} \rho_1^2 \zeta_{a,b;\varepsilon}^p f_p^{-3} z^2 + a^2 R^2 \int_{\U\cap \D_p} \zeta_{a,b;\varepsilon}^p f_p^{-1} |\nabla_t \rho_1|^2 z^2 \\
& \qquad + R^2 \int_{\U\cap \D_p} \zeta_{a,b;\varepsilon}^p f_p^{-1} \rho_1^2 |\nabla_t z|^2 + a^4 R^4 \int_{\U\cap \D_p} \rho_1^2 \zeta_{a,b;\varepsilon}^p f_p^{-3} z^2 + a^2 R^2 \int_{\U\cap \D_p} \zeta_{a,b;\varepsilon}^p f_p^{-1} |\nabla_x \rho_1|^2 z^2 \\
& \qquad + R^2 \int_{\U\cap \D_p} \zeta_{a,b;\varepsilon}^p f_p^{-1} \rho_1^2 |\nabla_x z|^2 + a^4 R^4 \int_{\U\cap \D_p} \rho_1^2 \zeta_{a,b;\varepsilon}^p f_p^{-3} z^2. \nt \label{eq_g19}
\end{align*}
Now, the seventh term on the right hand side containing $|\nb_x z|^2$ can be absorbed into the left hand side because $a \gg R$. For the terms containing $z^2$, we have
\begin{align*}
a^3 R^4 \rho_1^2 f_p^{-3} & + a^2 R^2 |\nabla_t \rho_1|^2 f_p^{-1} + 2 a^4 R^4 \rho_1^2 f_p^{-3} + a^2 R^2 |\nabla_x \rho_1|^2 f_p^{-1} \\
& \leqslant 3 a^4 R^4 \rho_1^2 f_p^{-3} + a^2 R^2 |\nabla_t \rho_1|^2 f_p^{-3} f_p^2 + a^2 R^2 |\nabla_x \rho_1|^2 f_p^{-3} f_p^2 \\
& \lesssim a^4 R^4 (\rho_1^2 + |\nabla_t \rho_1|^2 + |\nabla_x \rho_1|^2) f_p^{-3},
\end{align*}
where we used the fact that $f_p^2 < r^4 < R^4 < a^2 R^2$. Hence, \eqref{eq_g19} reduces to
\begin{align*}
aR^2 \int_{\U\cap \D_p} \rho_1^2 \zeta_{a,b;\varepsilon}^p f_p^{-1} | \nabla_x z|^2 & \lesssim \frac{1}{a} \int_{\U\cap \D_p} \rho_1^2 \zeta_{a,b;\varepsilon}^p f_p |\square z|^2 + R^2 \int_{\U\cap \D_p} \zeta_{a,b;\varepsilon}^p f_p^{-1} \rho_1^2 |\nabla_t z|^2\\
& \qquad + a^4 R^4 \int_{\U\cap \D_p} (\rho_1^2 + |\nabla_t \rho_1|^2 + |\nabla_x \rho_1|^2) \zeta_{a,b;\varepsilon}^p f_p^{-3} z^2.
\end{align*}
Now we will restrict the integral region on the left hand side to a subset. In particular, we consider the region where $\rho_1(t,x) = 1$, to get
\begin{align*}
aR^2 \int_{\mathcal{O}_{\sigma/2}(\Gamma_p^\varepsilon) \cap \U \cap \D_p } \zeta_{a,b;\varepsilon}^p f_p^{-1} | \nabla_x z|^2 & \lesssim \frac{1}{a} \int_{\U\cap \D_p} \rho_1^2 \zeta_{a,b;\varepsilon}^p f_p |\square z|^2 + R^2 \int_{W_p^\varepsilon} \zeta_{a,b;\varepsilon}^p f_p^{-1} |\nabla_t z|^2\\
& \qquad + a^4 R^4 \int_{W_p^\varepsilon} \zeta_{a,b;\varepsilon}^p f_p^{-3} z^2,
\end{align*}
where we used \eqref{eq:rho_1_MB} to get the appropriate integral regions for the terms on the right hand side. We have estimated the $|\nb_x z|^2$ term present on the right hand side of \eqref{eq:g2_MB}. Hence, we get
\[R \int_{ \Gamma_p^\varepsilon } \zeta_{a,b;\varepsilon}^p |\mathcal{N} z|^2 \lesssim \frac{1}{a} \int_{\U \cap \D_p} \zeta_{a,b;\varepsilon}^p f_p |\square z|^2 + a R^2 \int_{W_p^\varepsilon} \zeta_{a,b;\varepsilon}^p f_p^{-1} |\nabla_t z|^2 + a^4 R^4 \int_{W_p^\varepsilon} \zeta_{a,b;\varepsilon}^p f_p^{-3} z^2,\]
which concludes the proof.
\end{proof}

\begin{remark}
In the above theorem, we control the weighted norms of $(z,\nb_{t,x}z)$ in \eqref{eq.carleman_est} by terms involving the bulk: $\square z$, a first order term: $\nb_t z$, and a zeroth order term: $z$. Equivalently, we can replace the $\nb_t z$ on the right hand side with a $\nb_x z$ term. For this purpose, we only need to make slight modifications in the above proof. In particular, if we switch $t$ and $x$ after equation \eqref{eq:g2_MB} and continue the proof, we arrive at the analogous result with $\nb_x z$ on the right hand side of \eqref{eq.carleman_est}, instead of $\nb_t z$.
\end{remark}

\section{Uniqueness results} \label{sec_uni}
 
We are now going to use \cref{thm.carl_bdry} and \cref{thm_carl_int} to obtain suitable uniqueness results for the ultrahyperbolic equation. We choose the following constants
\begin{equation} \label{eq_sup_VX}
R_+:= \sup_{\U\cap\D_p} r_p, \qquad M_0 := \sup_{\mf{U}} |V|, \qquad M_1 := \sup_{\mf{U}} |\mc{X}^{t,x}|.
\end{equation}

\begin{theorem}[Boundary] \label{thm_uc_bdry_Dp}
Let $T>R_+$. Let $z_1, z_2$ be two solutions of \eqref{eq_main_z_A} and let $\Gamma_p^\varepsilon$ be given by \eqref{eq_def_Gamma_MB}. Then $\mc{N}z_1 = \mc{N}z_2$ on $\Gamma_p^\varepsilon$ implies that $z_1=z_2$ on $\mf{U}\cap \mf{D}_p$. That is, the value of $(\mc{N}z,z)|_{\Gamma_p^\varepsilon}$ uniquely determines the solution in the region $\mf{U}\cap \mf{D}_p$.
\end{theorem}
\begin{proof}[Proof of Theorem \ref{thm_uc_bdry_Dp}]
If $z_1$ and $z_2$ are two solutions of \eqref{eq_main_z_A}, then the function $z$ defined by $z:=z_1-z_2$ satisfies
\begin{equation} \label{eq_uni_main}
\begin{rcases}
\square z + \nabla_\mc{X} z + Vz = 0, \qquad & \text{in } \mf{U},\\
z = 0, \qquad & \text{on } \partial \mf{U}.
\end{rcases}
\end{equation}
Then to prove the result, it is enough to show that $z=0$ on $\p\U\cap\D_p$. Fix $\delta$ such that $0<\delta\ll 1$ and let $a$ be large enough such that
\begin{equation} \label{eq_a_choice}
a \gg \max \{ (m+n)^2, R_+, \delta^{-\frac{1}{3}} M_0^\frac{2}{3} R_+^\frac{4}{3}, \mu^{-1} \delta^{-2} M_1^2 R_+^4 \},
\end{equation}
where $\mu >0$ is to be chosen later.
Then choose the constants $\varepsilon$ and $b$ as follows
\begin{equation} \label{eq_varep}
\varepsilon := \delta^2 R_+^{-1}, \quad b := := \delta R_+^{-1}.
\end{equation}
Then we can apply \ref{thm.carl_bdry} and \eqref{eq.carleman_est} with the choice of constants mentioned above and $R:=R_+$, to get
\begin{align} \label{eq_uni_pf_1}
\notag \frac{C\delta^2}{R_+^2} & \int_{\mf{U}\cap \mf{D}_p}\zeta_{a,b;\varepsilon}^p (|u_p \partial_{u_p} z|^2 + |v_p \partial_{v_p} z|^2 + f_p g^{ab}\slashed\nabla_a^p z \slashed\nabla_b^p z - f_p g^{CD} \tilde{\nabla}_C^p z \tilde{\nabla}_D^p z ) + \frac{C\delta a^2}{R_+^2}\int_{\mf{U}\cap \mf{D}_p}\zeta_{a,b;\varepsilon}^p z^2 \\
& \leqslant \frac{1}{a}\int_{\mf{U}\cap \mf{D}_p} \zeta_{a,b;\varepsilon}^p f_p |\square z|^2 +  C' \int_{\Gamma_p^\varepsilon} \zeta_{a,b;\varepsilon}^p [( 1 - \varepsilon r_p ) \mc{N} f_p + \varepsilon f_p \mc{N} r_p ] |\mathcal{N} z|^2, 
\end{align}
where we have also used the fact that $0<f<r^2<R_+^2$ to get an extra power of $R_+$ in the denominators of the terms on the left hand side.
Since $\mathcal{N} z = 0$ on $\Gamma_p^\varepsilon$, it is enough to show that the first term present on the right hand side of the above estimate, namely
\[ \frac{1}{a}\int_{\mf{U}\cap \mf{D}_p} \zeta_{a,b;\varepsilon}^p f_p |\square z|^2 \]
can be absorbed into the left hand side. Using \eqref{eq_uni_main} to write the above term, we get
\[ \frac{1}{a}\int_{\mf{U}\cap \mf{D}_p} \zeta_{a,b;\varepsilon}^p f_p |\square z|^2 \leqslant \frac{2}{a}\int_{\mf{U}\cap \mf{D}_p} \zeta_{a,b;\varepsilon}^p f_p |\nb_\mc{X} z|^2 + \frac{2}{a}\int_{\mf{U}\cap \mf{D}_p} \zeta_{a,b;\varepsilon}^p f_p |V z|^2 =: I_1 + I_0.\]
For $I_1$, we have
\begin{equation} \label{eq_uni_pf_22}
I_1 = \frac{2}{a} \int_{\mf{U}\cap \mf{D}_p} \zeta_{a,b;\varepsilon}^p f_p |\nb_\mc{X} z|^2 = \frac{2}{a} \int_{\mf{U}\cap \mf{D}_p} \zeta_{a,b;\varepsilon}^p f_p | \mc{X}^\alpha \nb_\alpha z|^2.
\end{equation}
For $I_0$, using \eqref{eq_sup_VX} and \eqref{eq_a_choice}, shows that
\begin{align*}
I_0 :=  \frac{2}{a}\int_{\mf{U}\cap \mf{D}_p} \zeta_{a,b;\varepsilon}^p f_p |V|^2 |z|^2 \leqslant \frac{M_0^2 R_+^2}{a}\int_{\mf{U}\cap \mf{D}_p} \zeta_{a,b;\varepsilon}^p |z|^2 \ll \frac{\delta a^2}{R_+^2}\int_{\mf{U}\cap \mf{D}_p}\zeta_{a,b;\varepsilon}^p z^2.
\end{align*}
Thus, $I_0$ can be absorbed into the zeroth order term present on the left hand side of \eqref{eq_uni_pf_1}, after which we are left with 
\begin{equation} \label{eq_uni_pf_3}
C \int_{\mf{U}\cap \mf{D}_p}\zeta_{a,b;\varepsilon}^p \left[ \frac{\delta^2}{R_+^2} \left(|u_p \partial_{u_p} z|^2 + |v_p \partial_{v_p} z|^2 + f_p g^{ab}\slashed\nabla_a^p z \slashed\nabla_b^p z - f_p g^{CD} \tilde{\nabla}_C^p z \tilde{\nabla}_D^p z \right) + \frac{\delta a^2}{R_+^2} z^2 \right] \leqslant I_1.
\end{equation}
From the above estimate and \eqref{eq_uni_pf_22}, we see that the weights in the integrand of the first order term present on the left hand side of the above estimate are different from the weight in the integrand in $I_1$. Thus, we cannot absorb $I_1$ into the left hand side directly. For this purpose, we need to split the domain into two parts. Note that as mentioned in \eqref{assump_X}, $\mc{X}$ vanishes on $\U \cap \p\D_p$. This, along with the fact that $\X$ is a smooth function, implies that it must also vanish near the boundary of the cone $\p\D_p$, that is, the part where $f_p$ vanishes. Hence, there exists a $\mu>0$ such that 
\[ \mc{X} (t,x) = 0 \text{ on } \mf{U} \cap \mf{D}_p \cap \left\{ f_p \leq \mu \right\}. \]
Now we split the integral region as follows
\begin{align*}
\mf{U}_\leq &:= \mf{U} \cap \mf{D}_p \cap \left\{ f_p \leq \mu \right\}, \\
\notag \mf{U}_> &:= \mf{U} \cap \mf{D}_p \cap \left\{ f_p > \mu \right\}.
\end{align*}
Then $\mc{X}$ vanishes identically on $\U_\leqslant$. Next, note that on $\U_>$, we have
\begin{align*}
v_p = \frac{ f_p }{ - u_p } \gtrsim \frac{ \mu }{ R_+ } , \qquad - u_p = \frac{ f_p }{ v_p } \gtrsim \frac{ \mu }{ R_+ } .
\end{align*}
Using the above, we estimate the left hand side of \eqref{eq_uni_pf_3} from below as follows 
\begin{align*}
& C \int_{\mf{U}\cap \mf{D}_p} \zeta_{a,b;\varepsilon}^p \left[ \frac{\delta^2}{R_+^2}  \left(|u_p \partial_{u_p} z|^2 + |v_p \partial_{v_p} z|^2 + f_p g^{ab}\slashed\nabla_a^p z \slashed\nabla_b^p z - f_p g^{CD} \tilde{\nabla}_C^p z \tilde{\nabla}_D^p z \right) + \frac{\delta a^2}{R_+^2} z^2 \right] \\
& \geqslant \frac{C \mu \delta^2 }{R_+^3} \int_{\mf{U}_>}\zeta_{a,b;\varepsilon}^p \left(- u_p |\partial_{u_p} z|^2 + v_p |\partial_{v_p} z|^2 + v_p g^{ab}\slashed\nabla_a^p z \slashed\nabla_b^p z - v_p g^{CD} \tilde{\nabla}_C^p z \tilde{\nabla}_D^p z \right) \\
& \qquad + \frac{C\delta a^2}{R_+^2} \int_{\mf{U}\cap \mf{D}_p} \zeta_{a,b;\varepsilon}^p z^2.
\end{align*}
Combining \eqref{eq_uni_pf_3} and the above we get
\begin{align*}
\frac{C \mu \delta^2 }{R_+^3} & \int_{\mf{U}_>}\zeta_{a,b;\varepsilon}^p \left(- u_p |\partial_{u_p} z|^2 + v_p |\partial_{v_p} z|^2 + v_p g^{ab}\slashed\nabla_a^p z \slashed\nabla_b^p z - v_p g^{CD} \tilde{\nabla}_C^p z \tilde{\nabla}_D^p z \right) \\
+ & \frac{C\delta a^2}{R_+^2} \int_{\mf{U}\cap \mf{D}_p}\zeta_{a,b;\varepsilon}^p z^2 \leqslant I_1. \nt \label{eq_uni_pf_31}
\end{align*}
Now, from \eqref{eq_uni_pf_22}, we have
\begin{align*}
I_1 & = \frac{2}{a} \int_{\mf{U}_>} \zeta_{a,b;\varepsilon}^p f_p | \mc{X}^\alpha \nb_\alpha z|^2  + \frac{2}{a} \int_{\mf{U}_\leqslant} \zeta_{a,b;\varepsilon}^p f_p | \mc{X}^\alpha \nb_\alpha z|^2 \\
& \leqslant \frac{2 M_1^2 R_+}{a} \int_{\mf{U}_>} \zeta_{a,b;\varepsilon}^p  \left(-u_p| \partial_{u_p} z|^2 + v_p|\partial_{v_p} z|^2 + v_p g^{ab}\slashed\nabla_a^p z \slashed\nabla_b^p z - v_p g^{CD} \tilde{\nabla}_C^p z \tilde{\nabla}_D^p z \right) \\
& \ll C \frac{\mu \delta^2 }{R_+^3}  \int_{\mf{U}_>}\zeta_{a,b;\varepsilon}^p \left(-u_p |\partial_{u_p} z|^2 + v_p |\partial_{v_p} z|^2 + v_p g^{ab}\slashed\nabla_a^p z \slashed\nabla_b^p z - v_p g^{CD} \tilde{\nabla}_C^p z \tilde{\nabla}_D^p z \right),
\end{align*}
where we used \eqref{eq_a_choice} in the last step.
Using the above in \eqref{eq_uni_pf_31}, shows that the right hand side of \eqref{eq_uni_pf_31} can be absorbed into the left hand side, consequently giving us
\begin{align*}
\frac{C \mu \delta^2 }{R_+^3} & \int_{\mf{U}_>}\zeta_{a,b;\varepsilon}^p \left(|u_p \partial_{u_p} z|^2 + |v_p \partial_{v_p} z|^2 + f_p g^{ab} \slashed\nabla_a^p z \slashed\nabla_b^p z - f_p g^{CD} \tilde{\nabla}_C^p z \tilde{\nabla}_D^p z \right) \\
& + \frac{C\delta a^2}{R_+^2} \int_{\mf{U}\cap \mf{D}_p} \zeta_{a,b;\varepsilon}^p z^2 \leqslant 0.
\end{align*}
Then we drop the first order terms from the left hand side, since they are positive, and conclude that
\[ \frac{C\delta a^2}{R_+^2} \int_{\mf{U}\cap \mf{D}_p}\zeta_{a,b;\varepsilon}^p z^2 \leqslant 0, \]
which implies that $z = 0$ on $\mf{U}\cap \mf{D}_p$.
\end{proof}

Using the same ideas as the above theorem, we can prove the following result which addresses the uniqueness problem from interior data. Note that, we will now need to use Theorem \ref{thm_carl_int}.
\begin{theorem}[Interior] \label{thm_uc_int}
Let $T>R_+$. Let $z_1$ and $z_2$ be two solutions of \eqref{eq_main_z_A} and let $W_p^\varepsilon$ be given by \eqref{eq_def_W_pe}. Then, we have the following cases for uniqueness from interior data
\begin{enumerate}

\item  If $(z_1,\nb_t z_1 ) = (z_2,\nb_t z_2 )$ on $W_p^\varepsilon$, then $z_1=z_2$ on $\mf{U} \cap \mf{D}_p$.

\item  If $(z_1,\nb_x z_1 ) = (z_2,\nb_x z_2 )$ on $W_p^\varepsilon$, then $z_1=z_2$ on $\mf{U} \cap \mf{D}_p$.

\end{enumerate}
\end{theorem}
The only point where the proof of the above result differs from that of \cref{thm_uc_bdry_Dp} is that now instead of integrals over $\Gamma_p^\varepsilon$, we will now have integrals over $W_p^\varepsilon$. However, using the hypothesis, the integrals over $W_p^\varepsilon$ will vanish and the rest of the proof follows in a similar manner as \cref{thm_uc_int}.

\begin{remark} \label{rem_compare}
We will now give a comparison between the preliminary version of our result \cref{thm_prelim} and the precise versions \cref{thm_uc_bdry_Dp} and \cref{thm_uc_int}. The difference between $\Gamma_p$ and $\Gamma_p^\varepsilon$ can be made arbitrarily small. Indeed, the choice of $\varepsilon$ in \eqref{eq_varep} implies that
\[ ( 1 - \varepsilon r_p ) \mc{N} f_p + \varepsilon f_p \mc{N} r_p = \left( 1 - \frac{\delta^2}{R_+} r_p \right) \mc{N} f_p + \frac{\delta^2}{R_+} f_p \mc{N} r_p \xrightarrow{\delta \searrow 0} \mc{N} f_p. \]
Thus, we can choose sufficiently small $\delta$ such that $\overline{\Gamma_p^\varepsilon} \subset \Theta_p$, which is the region stated in \cref{thm_prelim}. Similarly, for the interior case we get $W_p^\varepsilon \subset W_p$. These set inclusions are proper subsets. Thus, the best uniqueness results are obtained from $\Gamma_p^\varepsilon$ (for boundary) and $W_p^\varepsilon$ (for interior).
\end{remark}

\section{Time-dependent domains} \label{sec_timedep}
In this section we show how the uniqueness results stated above can be generalised to time dependent domains. We only show the case when the change is along any one temporal coordinate. Without loss of generality, we assume that the domain is changing along the direction $t_1$.

\begin{definition}[Time dependent domain]
Let $T>0$ and let $G'$ be a hypercube of side $T$ in $\R^{m-1}$
\[ G':= \{ (t_2,\ldots,t_m) \in \R^{m-1} | -T < t_i < T, \text{ for } 2 \leqslant i \leqslant m \} \]
We will consider domains $\U \subset \Rmn$ of the following type. Assume that the boundary of $\U$, denoted by $\p\U$, is smooth and timelike. 
For each $t_1 \in (-T,T)$, we define $\Omega_{t_1}$ to be a non-empty, open, and bounded subset of $\R^n$ satisfying
\[\Omega_{t_1}:= \{ (x_1,\ldots,x_n) \in \R^n | \text{ for } (t_2,\ldots,t_m)\in G', \text{ we have } (t_1, t_2 \ldots, t_m,x_1,\ldots,x_n) \in \U \} \subset \R^n.\]
That is, we can write $\U$ as
\[ \mf{U} := \bigcup_{t_1\in (-T,T)} \{t_1\} \times G' \times \Omega_{t_1}. \]
Also assume that there exists a smooth future-pointing timelike vector field $\mc{Y}$ on $\bar{\U}$, such that $\mc{Y}|_{\bar{\U}}$ is tangent to $\p\U$. 
\end{definition}
Now suppose we are given a time dependent domain $\U$. Let us restrict ourselves to a particular $\Omega_{t_1}$, say $\Omega_0$. Denote the coordinate system on $\Omega_0$ by $(x_1,\ldots,x_n)$. Then we can transfer these coordinates to any level of $\U$ by moving them along the integral curves of $\mc{Y}$. With respect to the coordinate system $(t_1,t_2,\ldots,t_m,x_1,\ldots,x_n)$, we see that $\U$ can be equivalently written as $\U \simeq [-T,T] \times G' \times \Omega_0$. Hence, we have the transformation of the coordinate system on $\U$ given by $(\mc{Y},t_2,\ldots,t_m,x_1,\ldots,x_n)$ into the moving coordinate system on $[-T,T] \times G' \times \Omega_0$ given by $(t_1,t_2,\ldots,t_m,x_1,\ldots,x_n)$. Thus, we can re-parametrize $\U$ into a time static domain.

With the above definition and explanation, we can obtain the uniqueness results of the previous section, analogously. To avoid repetition, we do not write the statements or proofs again. The uniqueness results are precisely the ones given in \cref{thm_uc_bdry_Dp} and \cref{thm_uc_int}. Finally, comparing with \eqref{eq_N_static}, for the time dependent setting the corresponding calculation is
\begin{align*}
\mc{N} f_p = \frac{1}{4} \mc{N} (|x_p|^2 - |t_p|^2) = \frac{1}{4} & \left( \nu^{t_1} \p_{t_1} + \sum_{j=1}^n \nu^{x_j} \p_{x_j} \right) (|x_p|^2 - |t_p|^2) \\
& = \frac{1}{4} \left[ \sum_{j=1}^n \nu^{x_j} \p_{x_j} |x_p|^2 - \nu^{t_1} \p_{t_1} |t_p|^2 \right] = \frac{1}{2} ( \nu^x \cdot x_p - \nu^{t_1} \cdot t_{p,1}). \numberthis  \label{eq_N_moving}
\end{align*}

\section{Concluding remarks} \label{sec_conc}
It would be interesting to remove the assumption that $\X$ vanishes on the null cone. The main reason for this assumption is that our Carleman weight vanishes on the cone. A similar issue was resolved in \cite{MR3964826, MR4314050}, which address control problems for waves, by appropriately using energy estimates. This does not seem helpful here because we do not have any energy estimate for \eqref{eq_main_z_A}. It seems that one needs to use a different Carleman weight to solve the problem when $\X$ may not vanish on the cone.

Another direction is the problem of showing uniqueness inside the cone. For this purpose, we either need an energy estimate or we have to somehow show that the solution inside the cone is determined by the data on the cone. Both of these are undetermined till now. The only possible direction we have till now is that the function that we have used to define the Carleman weight satisfies the pseudoconvexity condition even inside the cone. Thus, it would be interesting to see if one can obtain a suitable Carleman estimate inside the cone.

Finally, we mention that the controllability of ultrahyperbolic equations is an unexplored area and to the best of the author's knowledge there are no results in this direction. As we mentioned earlier, the main obstacle is well-posedness.  In the presence of multiple time dimensions, even the concept of evolution of the system needs to be described properly. One idea is to consider evolution in the first time component only. Moreover, there needs to be a compatibility condition between the prescribed data, as normally this system is over-determined.

\bibliographystyle{siam}
	\bibliography{references}

\end{document}